\documentclass{article}
\usepackage{hyperref}
\usepackage{amsmath,amssymb,amsthm}
\usepackage{enumerate}
\usepackage{graphicx} 
\usepackage{caption}
\usepackage{subcaption}
\DeclareGraphicsExtensions{.pdf,.png}
\newtheorem{thm}{Theorem}[section]
\newtheorem*{thm*}{Theorem}
\newtheorem*{lem*}{Lemma}
\newtheorem{lem}[thm]{Lemma}
\newtheorem{cor}[thm]{Corollary}

\newtheorem{defn}[thm]{Definition} 
\newtheorem{conj}[thm]{Conjecture} 
\begin{document}

\title{The $3\times 3$ rooks graph ($K_3\square K_3$) is the unique smallest graph with lazy cop number 3}
\author{Brendan W. Sullivan\thanks{\href{mailto:sullivanb@emmanuel.edu}{\tt sullivanb@emmanuel.edu}; Dept. of Mathematics, Emmanuel College, 400 The Fenway, Boston, MA 02115}, Nikolas Townsend\thanks{\href{mailto:townsendn@emmanuel.edu}{\tt townsendn@emmanuel.edu}}, Mikayla Werzanski\thanks{\href{mailto:werzanskim@emmanuel.edu}{\tt werzanskim@emmanuel.edu}}} 
\date{\today}
\maketitle
\begin{abstract}
In the ordinary version of the pursuit-evasion game \textit{cops and robbers}, a team of cops and a robber occupy vertices of a graph and alternately move along the graph's edges, with perfect information about each other. If a cop lands on the robber, the cops win; if the robber can evade the cops indefinitely, he wins. In the variant \textit{lazy cops and robbers}, the cops may only choose one member of their squad to make a move when it's their turn. The minimum number of cops (respectively lazy cops) required to catch the robber is called the \textit{cop number} (resp. \textit{lazy cop number}) of $G$ and is denoted $c(G)$ (resp. $c_L(G)$). Previous work by Beveridge at al. has shown that the Petersen graph is the unique graph on ten vertices with $c(G)=3$, and all graphs on nine or fewer vertices have $c(G)\leq 2$. (This was a self-contained mathematical proof of a result found by computational search by Baird and Bonato.) In this article, we prove a similar result for lazy cops, namely that the $3\times 3$ rooks graph ($K_3\square K_3$) is the unique graph on nine vertices which requires three lazy cops, and a graph on eight or fewer vertices requires at most two lazy cops.
\end{abstract}

\section{Introduction}
Throughout, we work with finite, simple, undirected, connected graphs. The game \textit{cops and robbers} on such graphs was introduced by both Nowakowski \& Winkler \cite{nowakowski1983vertex} and Quilliot \cite{quilliot1978jeux} and has been studied extensively since then, leading to many deep conjectures and results as well as some interesting variations on the standard game. We highly recommend Bonato \& Nowakowski's book \textit{The Game of Cops and Robbers on Graphs} for an extensive survey of the current state of the field \cite{bonato2011game}.

\subsection{The standard game and known results}
Given a graph $G$ and some number $k$ of cops, the game plays as follows: 
\begin{itemize}
\item Each of the cops chooses a vertex on which to begin. 
\item In response, the robber chooses a vertex on which to begin. 
\item The cops' turn is first. Each cop may move along an edge of $G$ or choose to stay put. 
\item The robber's turn is next. He may move along an edge of $G$ or choose to stay put. 
\item The turns continue to alternate like this, with both sides having perfect information about the locations of all players.
\item The cops win if, at any point, a cop occupies the same vertex as the robber. 
\item Otherwise, the robber wins by indefinitely evading the cops.
\end{itemize}

For a given graph $G$, we seek its \textit{cop number}, denoted $c(G)$. This is the minimum number of cops required to guarantee the existence of a winning strategy whereby they catch the robber after finitely many moves. Such a number must exist because $c(G)\leq\gamma(G)$, the \textit{domination number} of $G$: the cops can win in one turn if they start on the vertices of a dominating set of $G$. 

Nowakowski \& Winkler first characterized the graphs which have $c(G)=1$, which they referred to as \textit{cop win} \cite{nowakowski1983vertex}. Shortly thereafter, Aigner \& Fromme proved two interesting results \cite{aigner1984game}. They showed that the class of planar graphs has bounded cop number: $c(G)\leq 3$ whenever $G$ is planar. By contrast, they showed that the class of all graphs has unbounded cop number: if the minimum degree satisfies $\delta(G)\geq k$ and $G$ has girth at least 5 (i.e. $G$ contains no 3- or 4-cycles), then $c(G)\geq k$. They followed this with an explicit construction of $k$-regular graphs containing no 3- or 4-cycles.

Consider the \textit{Petersen Graph} shown in Figure \ref{fig petersen}. Notice that it satisfies the hypotheses of Aigner \& Fromme's second result mentioned above: the graph is 3-regular and contains no 3- or 4-cycles. So, at least three cops are required to catch a robber on the Petersen Graph. Moreover, the domination number of the graph is three: indeed, the three vertices marked with boxes in the figure comprise a dominating set. Therefore, the cop number of the Petersen Graph is three.

\begin{figure}
\caption{The Petersen Graph is the unique smallest 3-cop win graph.}
\centering
\includegraphics[width=0.3\textwidth]{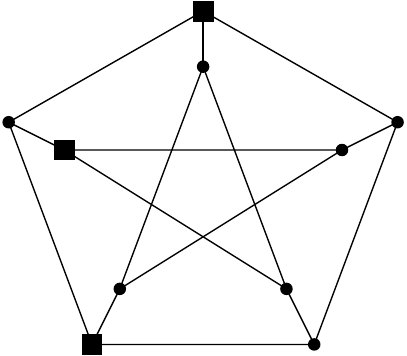}
\label{fig petersen}
\end{figure}

Recently, Baird et al. conducted an exhaustive computer search to find the cop numbers of graphs with few vertices. (An algorithm is contained in \cite{bonato2011game} which inputs a graph $G$ and a number of cops $k$ and returns {\sf True} if $c(G)\leq k$ and {\sf False} otherwise. That is, the algorithm can \textit{test} a graph's proposed cop number.) They concluded that the Petersen Graph is, in fact, the unique smallest graph with cop number three \cite{baird2013minimum}. That is, they found that (i) any graph on nine or fewer vertices has cop number at most two and (ii) amongst all graphs on ten vertices, only the Petersen Graph has cop number three while the rest have cop number one or two.

Shortly thereafter, Beveridge et al. posted an article that confirmed the aforementioned result via a self-contained mathematical proof \cite{beveridge2011petersen}. They approached the problem by proving a few facts relating the maximum degree of a graph to its cop number, including the fact that if a graph $G$ on $n$ vertices has maximum degree $\Delta(G)\geq n-5$, then $c(G)\leq 2$. (Compare this to our Lemma \ref{lem Delta n-4} below.) This helps to narrow the search to graphs on ten vertices. Thereafter, the proof that the Petersen Graph is unique required careful analysis of the strategies whereby the cops catch the robber. They prove a few lemmas that describe strategies for how two cops can catch a robber on a graph with ten vertices, and eventually deduce that the Petersen Graph is the only one not covered by those lemmas. They concluded with a conjecture that the $(k-5)$-cage solves this problem in general: that is, they conjecture that the $(k,5)$-cage is the unique smallest graph with cop number $k$. Of note, the Petersen Graph is the $(3,5)$-cage, and their conjecture is consistent with the outstanding Meyniel's Conjecture \cite{baird2013meyniel,beveridge2011petersen}.

\subsection{The lazy cops variant and known results}

Offner \& Ojakian investigated cops and robbers on the hypercubes $Q_n$ \cite{offner2014variations}. Specifically, they tweaked the rules so that some \textit{proportion} of the cops are allowed to move on each turn and then explored the tradeoff between that proportion and the number of cops required to win on $Q_n$. The extreme cases are where all cops are allowed to move (the ordinary version of the game) and where only one cop is allowed to move. Since then, that extreme case where only one cop is allowed to move has gained the name {\em lazy cops and robbers} in the literature \cite{bal2015lazy,bal2015lazy2,kinnersley2015cops}.

Analogously to $c(G)$, we define the \textit{lazy cop number} $c_L(G)$ to be the number of cops required to catch a robber on $G$ with the stipulation that only one cop is allowed to move on their turn. (The cops get to choose which of them moves on any given turn.) There is a simple and convenient relationship between $c$ and $c_L$, as stated in the following lemma.
\begin{lem}
\label{lem c cl gamma}
For any graph $G$,  $c(G)\leq c_L(G)\leq \gamma(G)$.
\end{lem}
The proof is simple: the first inequality follows because ordinary cops can adopt the strategy of lazy cops and simply choose not to make more moves, so they shouldn't need more than $c_L(G)$-many cops; the second inequality follows because placing lazy cops on the vertices of a dominating set allows them to win in one move, with only one cop needing to move to do so.

This variant of the original game has spurred some research into how $c$ and $c_L$ may differ for various classes of graphs. Bal et al. have studied lazy cops on hypercubes and on random graphs \cite{bal2015lazy,bal2015lazy2}. Kinnersley used lazy cops as part of a reduction to show that cops and robbers is EXPTIME-complete \cite{kinnersley2015cops}. The present authors of this article have investigated both ordinary and lazy cops who can move like standard chess pieces on boards of various sizes. Indeed, it was this work that led us to the result in this paper. Some of our other results have been submitted and some were presented at the Joint Mathematics Meetings in January 2016.

In general, there are plenty of open areas of research into this variation of the game and how it compares to the original. We list a few interesting questions: 
\begin{itemize}
\item What characterizes graphs $G$ for which $c(G)=c_L(G)$? For those graphs, can we say anything about \textit{how many more moves} it takes for the lazy cops to win? 
\item What characterizes graphs $G$ for which $c(G)<c_L(G)$? 
\item What characterizes graphs $G$ for which $c_L(G)=\gamma(G)$? 
\item Of all graphs $G$ with $n$ vertices, how many of them have $c(G)<c_L(G)$? 
\item Does Aigner \& Fromme's result about planar graphs carry over to the lazy cops variant? Or is there a planar $G$ such that $c_L(G)\geq 4?$ 
\end{itemize}

\section{Preliminaries}
\subsection{Definitions, notation}
Our main result concerns the particular graph $R_3=K_3\square K_3$, so we first define this graph. The notation ``$\square$'' is chosen because it is a visualization of the Cartesian product of $P_2$ (a path of length 2, i.e. an edge) with itself. 
\begin{defn}
Given graphs $G,H$, their \textit{Cartesian product}, denoted $G\square H$, is the graph with vertex set $V(G)\times V(H)$ and with an edge $(u,v)\sim (x,y)$ if and only if (i) $u=x$ in $G$ and $v\sim y$ in $H$, or (ii) $u\sim x$ in $G$ and $v=y$ in $H$.
\end{defn} 

We may think of this in terms of the cops and robber game when played on such a graph $G\square H$ as follows. Consider a \textit{position} (i.e. a vertex in the graph) as a point with two coordinates, one being a vertex of $G$ and one being a vertex of $H$. Consider a \textit{legal move} in the graph to consist of making either a move along an edge of $G$ or a move along an edge of $H$ but not both simultaneously. Compare this to the \textit{strong product}, denoted $G\boxtimes H$, where a legal move can be made along an edge of both $G$ and $H$ simultaneously. For more information about graph products, see e.g. \cite{imrich2000product,hammack2011handbook}.

\begin{defn}
We use $R_n$ to mean the \textit{$n\times n$ Rooks graph}. This is a graph with $n^2$ vertices arranged in $n$ rows and $n$ columns each of size $n$ such that every vertex is adjacent to each other vertex in its row and in its column. That is, the vertices of this graph are the squares of a standard $n\times n$ chessboard, and the edges represent the legal moves allowed by a Rook in standard chess. 
\end{defn}
Notice that, in fact, $R_n = K_n\square K_n$. The graph has $n^2$ vertices and an edge $(u,v)\sim (x,y)$ is present if and only if (i) $u=x$ and $v\neq y$, or (ii) $u\neq x$ and $v=y$. This corresponds exactly to the legal moves of a Rook in standard chess, thinking of the first coordinate as the Rook's row and the second coordinate as its column. 

Because of this correspondence, we sometimes find it convenient to refer to ``the board'' and ``rows/columns'' and other such terminology from chess, as opposed to graph theoretic terms. In particular, this makes the proofs in the following subsection easier to follow, we believe.

\subsection{Outline of results and techniques} 
In general, we find that $c_L(R_n)=\gamma(R_n)=n$ and yet $c(R_n)=2$ for all $n\geq 2$. Thus, we have a class of graphs for which the ordinary cop number is bounded yet the lazy cop number grows without bound. 
\begin{thm}
\label{thm Rn}
The $n\times n$ Rooks graph $G=K_n\square K_n$ has $c_L(G)=\gamma(G)=n$ and $c(G)=2$ (for $n\geq 2$).
\end{thm}
\begin{proof}
We know that $c(G)>1$ because of the characterization of cop-win graphs in \cite{nowakowski1983vertex}. We also see that two cops can win as follows: place them on opposite corners of the board and let the robber start anywhere he wishes. On the cops' first turn, send one of them to occupy the same row as the robber and send the other to occupy the same column as the robber. The robber is now trapped. 

That $\gamma(G)=n$ is clear. (For a more general proof that the domination number of an $m\times n$ board is $\min\{m,n\}$, see \cite{yaglom1987challenging}.) We now show that the robber can evade $n-1$ lazy cops, thus necessitating $c_L(G)=n$. No matter where these $n-1$ cops begin, the Pigeonhole Principle guarantees some row and some column that contain no cops; let the robber start at the square common to that row and column. Thereafter, the robber may choose to pass his turn if he is not under threat. If he is under threat, then \textit{only one cop is threatening} since only one may move at a time. Then, the robber looks in the direction perpendicular to this threat. There are $n-2$ other cops out there and $n-1$ squares available, so again we are guaranteed to find a safe square. The robber can use this strategy indefinitely. 
\end{proof}

The previous proof makes use of an important observation about the lazy cops variant: if the robber is ever under threat, then \textit{only one cop is threatening}. Otherwise, if the robber is under threat by two cops, this means one just moved to create a threat when there had already been one, in which case that threatening cop surely should have just caught the robber instead! We will not have occasion to use this idea in the remainder of this article, but we point it out here as a useful observation. (Indeed, we have used it fruitfully in regards to other results that have been submitted.)

We are now ready to state our main result: 
\begin{thm}
\label{thm R3}
The $3\times 3$ Rooks graph $G=K_3\square K_3$ is the unique graph on 9 vertices with $c_L(G)=3$. All other graphs $H$ on 9 vertices have $c_L(H)\leq 2$.
\end{thm}
We will prove this in the remaining sections of this paper by considering the maximum degree $\Delta$ of a graph. We start in Section \ref{sec large degree} by considering graphs on 8 or fewer vertices, as well as graphs on 9 vertices with $\Delta\geq 5$. In both cases, we find $c_L\leq 2$. We then show that 9 vertices and $\Delta=4$ makes $c_L\leq2$, with one notable exception: $R_3$ is 4-regular and has $c_L=3$.  We continue by considering graphs with $\Delta \leq 3$. This proved to be more challenging than we imagined, so we tackle this in two parts: Section \ref{sec 10 edges} concerns graphs with 9 vertices and at most 10 edges; Section \ref{sec cases} concerns graphs with 9 vertices and 11 to 13 edges. This restriction on the number of edges comes from $\Delta\leq 3$, as well as the following result about the minimum degree of a graph.

\begin{lem}[$\delta\geq 2$ Suffices]
\label{lem delta 2}
Assume $G=(V,E)$ has a vertex $v\in V$ with $\deg(v)=1$; say $uv\in E$ is the unique edge incident to $v$. Define $G'$ to be the graph with vertex set $V'=V-\{v\}$ and edge set $E'=E-\{uv\}$. Then $c_L(G')=c_L(G)$.
\end{lem}
\begin{proof}
Notice that $G'$ is, in fact, a \textit{retract} of $G$ since the map $\varphi:V\to V'$ defined by $\varphi(x)=x$ for all $x\in V'$ and $\varphi(v)=u$ preserves adjacencies in $G$. Berarducci \& Intrigila proved several helpful results about retracts, and we note that all of their arguments carry over from ordinary cops to lazy cops \cite{berarducci1993cop}. (For instance, they show that $c(H)\leq c(G)$ when $H$ is a retract of $G$ by arguing that the cops may catch the robber who plays on $H$ by acting as if he moves in $G$ and placing themselves according to the retraction map onto $H$.) 

With their results thus adapted to lazy cops, we may deduce that
\[
c_L(G') \leq c_L(G) \leq \max\{c_L(G'),2\}
\]
since $G-G'$ is the lone vertex $v$ which has cop number 1.  Thus, we find that $c_L(G')=c_L(G)$ provided $c_L(G')\geq 2$ so that it witnesses the maximum on the right-hand side. Otherwise, $c_L(G')=1$ which means it is \textit{dismantlable} in the sense given by Nowakowski \& Winkler \cite{nowakowski1983vertex}. Notice that $v$ is a \textit{pitfall} in $G$ since its closed neighborhood is dominated by $u$. So, we may begin to dismantle $G$ by removing $v$. But this yields precisely $G'$ which, as assumed, can be dismantled completely. Thus, if $c_L(G')=1$, then so is $c_L(G)=1$. In either case, we have $c_L(G')=c_L(G)$. 
\end{proof}
By virtue of this lemma, we may ignore graphs that have a vertex of degree 1. By removing that vertex, we obtain a graph with the same lazy cop number that is \textit{smaller}. So, when considering graphs on 9 vertices and already knowing that $c_L\leq 2$ for graphs on 8 vertices, we can specify that $\delta \geq 2$.

\section{Graphs with ``large'' maximum degree}
\label{sec large degree}
\subsection{$|V|\leq 8 \implies c_L\leq 2$}
\begin{thm}
\label{thm 8 vertices}
If $G$ is a connected graph on at most 8 vertices, then $c_L(G)\leq 2$.
\end{thm}
In this section, we prove \textbf{Theorem \ref{thm 8 vertices}}. We do this by taking $G=(V,E)$ with $|V|\leq 8$ and considering a vertex $u$ such that $\deg(u)=\Delta$, the \textit{maximum degree} of $G$. If $\Delta\geq 4$, then two lazy cops can win easily since one cop can dominate most of the graph single-handedly. If $\Delta =3$, then two lazy cops also win but we have one interesting case to consider. And if $\Delta=2$, then $G$ is a cycle and we know $c_L=c=2$. The following general result will be useful and is of interest on its own. (It is also an analogue to Corollary 1.4 of Beveridge et al.'s paper \cite{beveridge2011petersen} which states that $\Delta\geq n-5$ implies $c(G)\leq 2$.)
\begin{lem}
\label{lem Delta n-4}
For a graph $G$ on $n$ vertices with $\Delta\geq n-4$, we have $c_L(G)\leq 2$.
\end{lem}
\begin{proof}
Let $u\in V$ have $\deg(u)=\Delta\geq n-4$. Consider the induced subgraph external to $u$ and its neighbors, denoted $H:=G[V-\overline{N}(u)]$. We have two lazy cops at our disposal, and we will choose to initially place one at $u$ and keep him there, as this restricts the robber to playing on $H$ throughout the game.
\begin{itemize} 
\item If $H$ is empty then, in fact, $c_L(G)=1$.
\item If $H$ has one vertex, we may start the second cop there and the robber is certainly caught within one move. 
\item If $H$ has two vertices, say $a$ and $b$, we may start the second cop on, say, $a$. If $ab$ is an edge, the robber is caught. Otherwise, the robber was forced to start at $b$ and he cannot safely move. This second cop who started at $a$ may travel through the graph to reach $b$ (since $G$ is connected). 
\item If $H$ has three vertices, say $a,b,c$, we choose to start the second cop at the one that has the maximum degree within $H$, say $a$. (By \textit{degree within $H$} we mean to only consider adjacencies amongst $\{a,b,c\}$.)
\begin{itemize}
\item If $\deg_H(a)=2$, then the robber is caught immediately. 
\item If $\deg_H(a)=1$ and $ab$ is that edge, then the robber is forced to start at $c$ and he cannot safely move. Analogously to the previous case, this second cop can chase him down since $G$ is connected. 
\item If $\deg_H(a)=0$, then the robber is forced to start at $b$ or $c$ and cannot safely move. This second cop chases him down since $G$ is connected.
\end{itemize}
\end{itemize} 
Thus, two (or one) lazy cops win when $\Delta\geq n-4$.
\end{proof} 

We now know that when $|V|\leq 8$ and $\Delta\geq 4$, we have $c_L\leq 2$. As mentioned above, if $\Delta=2$ then $G$ is a cycle (since we also assume $\delta\geq 2$, by Lemma \ref{lem delta 2}), which has $c_L=2$. So, the only case remaining is when $\Delta=3$. 

\begin{lem}
\label{lem v8}
For a graph on $n\leq 8$ vertices with $\Delta=3$, $c_L\leq 2$.
\end{lem}
\begin{proof}
Consider a vertex $u$ with $\deg(u)=3$. As in the proof of Lemma \ref{lem Delta n-4}, we consider the induced subgraph external to $u$ and its neighbors, denoted $H:=G[V-\overline{N}(u)]$. We place one cop at $u$ initially so that the robber will be forced to start on a vertex in $H$.

This subgraph $H$ has at most 4 vertices; if $H$ has 3 or fewer vertices, then we may apply exactly the same arguments as those given in the proof of Lemma \ref{lem Delta n-4}. So, we need only consider the case where $H$ has 4 vertices, say $a,b,c,d$. We consider the degrees of these vertices within $H$.

\begin{itemize}
\item Let $A\subseteq \{a,b,c,d\}$ be the set of vertices that have degree 0 within $H$. If $A\neq\varnothing$, then we may place the second cop to force the robber to start somewhere in $A$. Then, we keep the first cop at $u$ which forces the robber to stay put; the second cop then chases down the robber, since $G$ is connected. 

\item If any of $\{a,b,c,d\}$ have degree 3 within $H$, then we place the second cop on such a vertex to dominate the entire graph $G$, thereby catching the robber within one turn. 

So, we only need to consider the cases where each vertex in $H$ has degree 1 or 2 within $H$. This means that the subgraph $H$ is either a path of length 4, or two disjoint edges, or a 4-cycle.

\item If $H$ is a path of length 4, start the second cop anywhere in $H$. By keeping the first cop at $u$, we can use the second cop to catch the robber on $H$. 

\item If $H$ is two disjoint edges, start the second cop anywhere in $H$. The robber is forced to start on a vertex incident to the other edge of $H$, say $ab$. By keeping the first cop at $u$, the robber is forced to be on either $a$ or $b$. We move the second cop through the graph to chase down the robber. 

\item So, assume $H$ is a 4-cycle: $a-b-c-d-a$. Because $G$ is connected, there must be at least one vertex of $H$, say $a$, that has an external neighbor, say $a'\in N(u)$. Since $\Delta=3$, we know this is the \textit{only} external neighbor of $a$. We choose to start the second cop at $c$ (the vertex opposite $a$ on the cycle). This forces the robber to start at $a$. We then send the first cop from $u$ to $a'$. This traps the robber since each of his only three neighbors are either occupied by the first cop or guarded by the second cop. 
\qedhere 
\end{itemize}
Lemmas \ref{lem Delta n-4} and \ref{lem v8} together prove Theorem \ref{thm 8 vertices}: $c_L\leq 2$ when $|V|\leq 8$.

\end{proof}

\subsection{$|V|=9$ and $\Delta\geq 4$}
\label{sec D4 R3}
From the previous section and Theorem \ref{thm Rn}, we know that the smallest graph(s) with $c_L=3$ has (have) 9 vertices. We now begin our work towards showing that $R_3=K_3\square K_3$ is, in fact, the unique such graph. 

Specifically, in this section, we work with graphs on 9 vertices with ``large'' maximum degree, which in this context means $\Delta\geq 4$. We will show that such graphs either have $c_L\leq 2$ or are forced to be $R_3$ itself. Luckily, some of the work towards this result has been achieved already. If $\Delta\geq 5$, then Lemma \ref{lem Delta n-4} applies and so $c_L\leq 2$.

If $\Delta=4$, then we follow an argument similar to that in the proof of Lemma \ref{lem v8}: we take a vertex $u$ with $\deg(u)=4$ and consider the induced subgraph external to $u$ and its neighbors, denoted $H:=G[V-\overline{N}(u)]$. We place one cop at $u$ initially so that the robber will be forced to start on a vertex in $H$. Notice that $H$ has 4 vertices and almost (but not all) of the arguments from the proof of Lemma \ref{lem v8} apply:
\begin{itemize}
\item If any vertices in $H$ have degree 0 within $H$, then we can place the second cop initially to force the robber to start on one of those vertices. He is then forced to stay put and we can chase him down with the second cop. 
\item If any of the vertices of $H$ have degree 3 within $H$, then we place the second cop on one of those vertices. Now, the two lazy cops dominate the entire graph so the robber is caught immediately. 

So, we only need to consider the cases where each vertex in $H$ has degree 1 or 2 within 4, meaning that $H$ is either a path of length 4, two disjoint edges, or a 4-cycle.
\item If $H$ is a path of length 4 or two disjoint edges, then the same arguments apply: we can start the second cop on $H$ and use him to chase down the robber (either within $H$ or by temporarily leaving $H$).
\item However, the case where $H$ is a 4-cycle is potentially troublesome. If we keep the first cop at $u$, then the robber can evade the second cop on the 4-cycle. And if we attempt to move the first cop from $u$ towards $H$, then we may run into a different problem: a vertex $a\in H$ may have \textit{two neighbors} external to $H$, which means the first cop, having moved from $u$, may not be able to guard \textit{both} of those neighbors. 
\item At least, if $H$ is a 4-cycle and some $a\in H$ has 0 or 1 neighbors \textit{external} to $H$, then the previous arguments do indeed apply. We can position the second cop to force the robber to start on that $a$ and keep him at bay. The first cop may then chase down the robber along the shortest path through $G$ from $u$ to $a$. (If $a$ has 1 external neighbor in $N(u)$, then the first cop simply moves there. Otherwise, he travels through $G$ while the robber is forced to stay put.)
\end{itemize}
So, in fact, we realize that the only worrisome situation is where $H$ is a 4-cycle and \textit{every} vertex in $H$ has \textit{two neighbors external to $H$}. This situation is the content of the lemma below.

\begin{lem}
Assume $G$ has 9 vertices, $\Delta=4$, and $u$ is a vertex with $\deg(u)=4$. Define $H$ to be the induced subgraph $H:=G[V-\overline{N}(u)]$ and suppose $H$ is a 4-cycle. Further, suppose that every vertex in $H$ has two neighbors external to $H$ (i.e. in $N(u)$). Then, either $c_L\leq 2$ or $G\cong R_3$.
\end{lem}
\begin{proof} 
Using the assumptions of the lemma, as well as some movements of two lazy cops, we will show that the only way to avoid a scenario where two lazy cops can win forces the creation of the specific graph $R_3$.

Label the vertices of $H$ as the 4-cycle $a-b-c-d-a$, and label the 4 neighbors of $u$ as $v,w,x,y$. We use the notation $N'(a)$ to mean, for instance, the neighborhood of $a$ external to $H$, i.e. $N'(a)=N(a)\cap N(u)$. We now begin with some helpful observations about these neighborhoods and (potential) adjacencies amongst $v,w,x,y$.
\begin{enumerate}[Cl{a}{i}m 1.]
\item Opposing vertices on the 4-cycle comprising $H$ must have disjoint external neighborhoods; that is, $N'(a)\cap N'(c)=N'(b)\cap N'(d)=\varnothing$.

\textit{Proof of Claim 1:} Assume for sake of contradiction that, say, $a$ and $c$ share a common external neighbor, say $v$. Start the cops at $u$ and $a$, forcing the robber to start at $c$. Without loss of generality, we can say $c$'s other external neighbor is $w$. (See Figure \ref{fig v9-c1}.) Move the cop from $u$ to $w$ to threaten the robber, who is now trapped: his only unoccupied neighbors (namely $b,d,v$) are all guarded by the cop at $a$, meaning two lazy cops could win. (The same argument applies to $b$ and $d$.)

Since each of $\{a,b,c,d\}$ has two external neighbors, we may further deduce that $N'(a)$ and $N'(c)$ partition $\{v,w,x,y\}$, and the same holds for $N'(b)$ and $N'(d)$.

\begin{figure}[h]
\caption{Opposing vertices on the 4-cycle cannot have a common neighbor.}
\centering 
\includegraphics[width=0.4\textwidth]{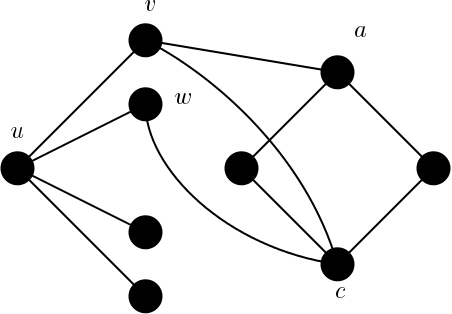}
\label{fig v9-c1}
\end{figure}

\item For any two vertices amongst $\{v,w,x,y\}$, if they share a neighbor in $H$, then they cannot be adjacent in $G$.

\textit{Proof of Claim 2:}  Assume for sake of contradiction that, say, $v$ and $w$ share a common neighbor in $H$, say $a$. (See Figure \ref{fig v9-c2}.) Start the cops at $u$ and $c$, forcing the robber to start at $a$. Move the cop from $u$ to $v$ to threaten the robber, who must now move to $w$. If $vw\in E$, then the robber is caught, meaning two lazy cops could win. (The same argument applies to any pair of vertices from $N(u)$.)

The contrapositive of this claim will be useful below, so we state it here: If $vw\in E$, then $v$ and $w$ have no common neighbor in $H$.

\begin{figure}[h]
\caption{If $v,w$ have a common neighbor, they cannot be adjacent.}
\centering 
\includegraphics[width=0.4\textwidth]{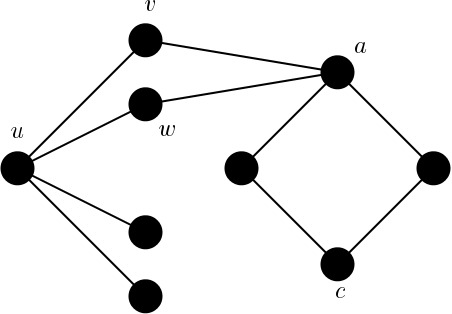}
\label{fig v9-c2}
\end{figure}

\item Assuming Claims 1 and 2 above, then there must be at least two edges amongst $\{v,w,x,y\}$.

\textit{Proof of Claim 3:} Assume for sake of contradiction that there are no edges amongst $\{v,w,x,y\}$. (See Figure \ref{fig v9-c3a}.) Start the cops at $u$ and $c$, forcing the robber to start at $a$. Without loss of generality, suppose $N'(a)=\{v,w\}$.  Move the cop from $u$ to $v$, forcing the robber to move to $w$. 

Now, apply the conclusion of Claim 1: \textit{exactly} one of $\{b,d\}$ is adjacent to $w$; suppose it's $b$. Move the cop from $c$ to $b$, threatening the robber who is now trapped at $w$: his only unoccupied neighbors (namely $a,u$) are both guarded by the cop at $v$. 

\begin{figure}[h]
\caption{When there are 0 internal edges amongst $\{v,w,x,y\}$.}
\centering 
\includegraphics[width=0.4\textwidth]{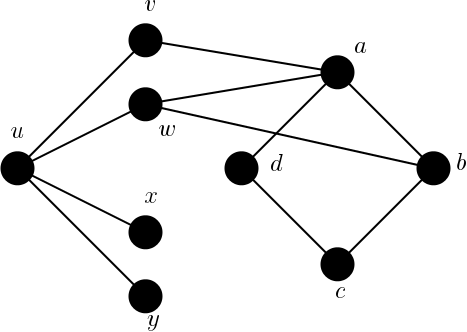}
\label{fig v9-c3a}
\end{figure}

Next, assume for sake of contradiction that there is one edge amongst $\{v,w,x,y\}$; without loss of generality, suppose that edge is $xy$. (See Figure \ref{fig v9-c3b}.) We will show that the cops can position themselves to force a situation like the one described in the previous paragraph. 

\begin{figure}[h]
\caption{When there is 1 internal edge amongst $\{v,w,x,y\}$.}
\centering 
\includegraphics[width=0.4\textwidth]{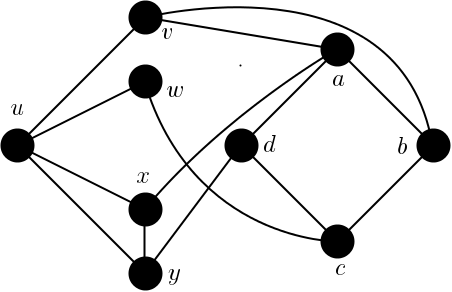}
\label{fig v9-c3b}
\end{figure} 

Start the cops at $u$ and $c$, forcing the robber to start at $a$. Now, consider $N'(a)$: it cannot be $\{x,y\}$, since the existence of the edge $xy$ would violate Claim 2 above; it also cannot be $\{v,w\}$ since Claim 1 would imply that $N'(c)=\{x,y\}$, which then violates Claim 2. Thus, $a$ must be adjacent to exactly one of $\{v,w\}$ and exactly one of $\{x,y\}$; without loss of generality, let's say $N'(a)=\{v,x\}$.

Send the cop from $u$ to $x$. This threatens the robber and forces him to move to $v$. As described in the scenario above, we move the other cop from $c$ to whichever of $\{b,d\}$ is adjacent to $v$; suppose it's $b$, as in Figure \ref{fig v9-c3b}. The robber is now trapped since his neighbors $u$ and $a$ are guarded by the cop at $x$.
\end{enumerate}

Thus, the only way to possibly avoid having $c_L(G)\leq 2$ requires two edges amongst $\{v,w,x,y\}$. In fact, these two edges must be disjoint, as we will now show. Since each of $\{a,b,c,d\}$ has exactly two neighbors in $\{v,w,x,y\}$ and, by Claim 1, $N(a),N(c)$ and $N(b),N(d)$ form partitions of $\{v,w,x,y\}$, we may deduce that, in fact, each of $\{v,w,x,y\}$ has two neighbors in $\{a,b,c,d\}$. Therefore, none of $\{v,w,x,y\}$ can acquire two more neighbors because this would violate $\Delta=4$. Since there must be two edges amongst $\{v,w,x,y\}$, this implies those two edges must be disjoint, i.e. $vw,xy\in E$ or $vy,wx\in E$.

Let's say $vw,xy\in E$ are those edges. We will now show that the only way to add the remaining edges to the graph, while obeying the results of the above claims, either creates a particular graph with $c_L=2$ or else creates an isomorphic copy of $R_3$. 

As shown above, Claims 1 and 2 together imply that each vertex in $H$ is adjacent to exactly one of $\{v,w\}$ and exactly one of $\{x,y\}$. Let's say $N'(a)=\{v,x\}$, which then forces $N'(c)=\{w,y\}$, by Claim 1. We now have two cases, dependent on whether the other neighborhoods $N'(b),N'(d)$ also equal $\{v,x\},\{w,y\}$, in some order. 

\begin{enumerate}[C{a}se 1.]
\item Suppose that $N'(b),N'(d)$ partition $\{v,w,x,y\}$ into $\{v,x\}$ and $\{w,y\}$. Without loss of generality, let's say $N'(b)=\{v,x\}$ and $N'(d)=\{w,y\}$. This particular graph has $c_L=2$. Start the cops at $u$ and $c$, forcing the robber to start at $a$. Move the cop from $u$ to $x$, threatening the robber and forcing him to move to $v$. Move the other cop from $c$ to $w$, threatening the robber who is now trapped: his three neighbors $a,b,u$ are all guarded by the cop at $x$.

\item Suppose that the neighborhoods $N'(b),N'(d)$ partition $\{v,w,x,y\}$ in a different way; in fact, the partition must be $\{v,y\},\{w,x\}$. We have two cases based on $N'(b)$; both yield an isomorphic copy of $R_3$.

\begin{enumerate}[C{a}se {2}a.]
\item Suppose $N'(b)=\{v,y\}$ and $N'(d)=\{w,x\}$. This graph is isomorphic to $R_3$ since we can arrange the vertices into three rows and columns such that two vertices are adjacent if and only if they belong to the same row or column: 
\[
\begin{array}{ccc}
u & v & w \\
x & a & d \\
y & b & c
\end{array}
\]
\item Suppose $N'(b)=\{w,x\}$ and $N'(d)=\{v,y\}$. This graph is also isomorphic to $R_3$: 
\[
\begin{array}{ccc}
u & v & w \\
x & a & b \\
y & d & c
\end{array}
\]
\end{enumerate}
\end{enumerate}
This completes the proof. 
\end{proof}
We conclude that with 9 vertices and $\Delta=4$, either two lazy cops can win or else the graph is $R_3$. We note that, in fact, $R_3$ is 4-regular. Given the Beveridge et al. result about the Petersen Graph, the conjecture they made about the $(k,5)$-cages, and this result we have so far, we strongly suspect that there is a close relationship between the cop number and regularity of a graph. 

Essentially, we have found that graphs where all the vertices have ``large'' degree are good for the cops because they can cover more territory simultaneously and therefore win within one or two moves. Likewise, we will find in the next section that graphs where all the vertices have ``small'' degree are good for the cops because the robber does not have many escape routes so the cops can corner him (although it may take several moves). A regular graph whose regularity is neither too large nor too small seems to strike a balance between these two phenomena. Obviously, these statements are somewhat vague (indeed, what exactly is ``large enough''?) but they are based on our extensive study of both ordinary and lazy cops. We believe there is something significant at play here; we just don't know exactly what. Indeed, we empathize with Riemann: ``If only I had the theorems! Then I should find the proofs easily enough.'' \cite{lakatos2015proofs}

\section{Graphs with ``small'' maximum degree}
In this section, we continue to work with graphs on 9 vertices and now narrow our focus on those with $2\leq \delta\leq\Delta\leq 3$. Since we work with connected, simple graphs, this necessitates $8\leq |E|\leq 13$. We will show that any such graph has $c_L\leq 2$ by separately considering graphs for which $8\leq |E|\leq 10$ and for which $11\leq |E|\leq 13$. Specifically, we will handle 8 to 10 edges quickly (in Section \ref{sec 10 edges}) since these graphs either have known lazy cop number or can be analyzed easily. Handling 11 to 13 edges (in Section \ref{sec cases}) will amount to seven successive scenarios, based on taking two vertices of degree three and considering how many neighbors they have in common. Each scenario allows us to either find a way for two lazy cops to win or reduce the scenario to a previously handled one.

This overall approach is effective but we believe the results here are a bit longer than they need to be. Really, we struggled to find a unifying principle amongst these graphs and how the cops manage to win on them. As we described above, it feels like low degree counts should benefit the cops since the robber has few ``escape routes'' no matter where he is. Indeed, we even used {\tt SageMath} to enumerate and show all 147 such graphs (for $10\leq |E|\leq 13$; 8 edges implies a tree and 9 edges implies a cycle, for both of which the lazy cop number is known) and found that, for each one, we could easily and quickly confirm by eye how two lazy cops could win. However, this sometimes requires the cops to make several moves, as opposed to the results in Section \ref{sec large degree}, where we typically made one cop stay put to guard a large portion of the graph from his high-degree vertex.

We believe this is why the results in this section were only achievable by exhaustive (and exhausting) analysis, but we have hope that there is another way of looking at it that makes the main result of this section appear more immediately. As we describe in Section \ref{sec summary}, this is of particular interest when we consider generalizing the overall result of this article, conjecturing that the unique smallest graph with $c_L=n$ is $R_n=K_n\square K_n$. Any progress towards this conjecture will require a better, more coherent analysis of graphs with ``small'' degree counts. 

\subsection{$8\leq |E|\leq 10 \implies  c_L\leq 2$} 
\label{sec 10 edges}
\begin{proof} A graph with 9 vertices and 8 edges is a tree $T$, and it is known that $c(T)=c_L(T)=1$ \cite{nowakowski1983vertex}. A graph with 9 vertices and 9 edges is the cycle $C_9$, and it is known that $c(C_n)=c_L(C_n)=2$ (for any $n\geq 4$) \cite{bonato2011game}. 

A (connected) graph with 9 vertices, 10 edges, and $2\leq\delta$ must be, essentially, two conjoined cycles. Observe that there must be exactly two vertices of degree 3, say $u$ and $v$, while the rest have degree 2. By taking each of $u$'s neighbors and following the subsequent paths they induce until reaching $v$ or else returning to $u$, we find that the graph is either: (i) two disjoint cycles, one containing $u$ and the other containing $v$, with a path from $u$ to $v$ connecting them; or (ii) three disjoint paths from $u$ to $v$.

In either case, we choose to start one cop at $u$ and another at $v$. The robber starts anywhere he can. If the graph is of variety (ii), the cop at $v$ travels along the path to $u$ that the robber occupies, so he will be caught. If the graph is of variety (i), we can do the same thing if the robber is on the path from $u$ to $v$. Otherwise, the robber is on one of the disjoint cycles, let's say the one that contains $u$. In that case, we fix the cop at $u$ to prevent the robber from leaving the cycle, meanwhile chasing him down with the other cop. 
\end{proof}

\subsection{For $|E|\geq 11$, scenarios that imply $c_L\leq 2$} 
\label{sec cases}
Throughout this section, we assume $G$ is connected on 9 vertices with $2\leq\delta\leq\Delta\leq 3$. This implies the existence of (at least) two vertices of degree 3. Each lemma in this section considers how many neighbors those two degree 3 vertices have in common. We show that each successive scenario either leads to two lazy cops winning, or else a previous scenario which has already been handled by 2 cops.  By the end of this section, we will have exhausted all possible scenarios. 
\begin{lem}
\label{lem case 1}
If there exist vertices $u_1,u_2$ both with degree 3 such that $\overline{N}(u_1)\cap \overline{N}(u_2)=\varnothing$, then $c_L(G)\leq 2$. 
\end{lem}
\begin{proof} 
This scenario implies that there is \textit{exactly one} vertex $a$ outside of the eight vertices comprising $\overline{N}(u_1)\cup \overline{N}(u_2)$. Start one cop at $u_1$ and another at $u_2$, forcing the robber to start at $a$.

\begin{figure}[h]
\caption{When $u_1,u_2$ have no neighbors in common (including each other).}
\centering 
\includegraphics[width=0.4\textwidth]{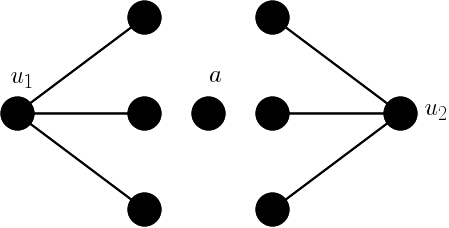}
\label{fig lemma-4-1}
\end{figure}  

Consider $a$'s neighbors, of which there are at most 3. Specifically, consider $N_1:=N(a)\cap N(u_1)$ and $N_2:=N(a)\cap N(u_2)$. By the Pigeonhole Principle, one of these sets has at most 1 element; let's say that set is $N_1$. Then we choose to station the cop at $u_2$, preventing $R$ from moving into $N(u_2)$. Meanwhile, we send the cop at $u_1$ towards the robber along the shortest path from $u_1$ to $v$: if $N_1=\{w\}$, say, then we send him along the path $u_1-w-v$; if $N_1=\varnothing$, then we send him along any path in $G$ (since it is connected), knowing that the robber cannot, in fact, move safely. See Figure \ref{fig lemma-4-1}.
\end{proof}

\begin{lem}
\label{lem case 2}
If there exist vertices $u_1,u_2$ both with degree 3 such that $\overline{N}(u_1)\cap \overline{N}(u_2)=\{z\}$ (where $z\neq u_1,u_2$), then $c_L(G)\leq 2$. 
\end{lem}
\begin{proof} 
This scenario implies that there are \textit{exactly two} vertices, say $a$ and $b$, outside of the seven vertices comprising $\overline{N}(u_1)\cup\overline{N}(u_2)$. We have two cases, based on the potential presence of the edge $ab$.

\begin{enumerate}
\item Suppose $ab\notin E$. Start the cops at $u_1$ and $u_2$, forcing the robber to start at either $a$ or $b$. Whichever he chooses (let's say $a$), look at its neighbors: some cop must be guarding all but one (or possibly all) of those neighbors. The only way to avoid this would be to have $a$ adjacent to $u_1$'s two neighbors outside of $N(u_2)$ as well as to $u_2$'s two neighbors outside of $N(u_1)$, which would mean $\deg(a)=4$. 

\begin{figure}[h]
\caption{When $u_1,u_2$ have exactly one neighbor in common (and not each other) and $ab\notin E$.}
\centering 
\includegraphics[width=0.4\textwidth]{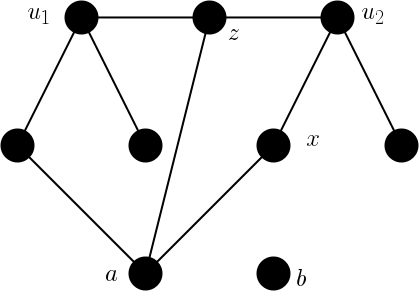}
\label{fig lemma-4-2-a}
\end{figure} 

Let's say the cop with that property is the one at $u_1$. (Note: it could be that both cops have this property, in which case we just pick one arbitrarily.) Keep him on his current vertex to effectively prevent the robber from moving. If $u_1$ guards \textit{all} of $a$'s neighbors, we send the cop at $u_2$ along any path to $a$, catching the robber. If $u_1$ does not guard $a$'s neighbor $x$, say, then we specifically send that cop along the particular path $u_2-x-a$. See Figure \ref{fig lemma-4-2-a}. 

\item Suppose $ab\in E$. Each of $\{a,b\}$ is guaranteed to have 1 or 2 other neighbors. First, suppose one of them has only 1 other neighbor, or has 2 neighbors that are both adjacent to $u_1$ (or, without loss of generality, $u_2$). Figure \ref{fig lemma-4-2-b} depicts this scenario, where 1 or 2 of the three dotted edges are present. Start the cops at $u_1$ and $u_2$, forcing the robber to start on $a$ or $b$. 

\begin{figure}[h]
\caption{When $u_1,u_2$ have exactly one neighbor in common (and not each other) and $ab\in E$ and $a$'s neighbors (besides $b$) are all adjacent to $u_1$. (Note: Only one or two of the three dotted edges are present, since $\Delta=3$.)}
\centering 
\includegraphics[width=0.4\textwidth]{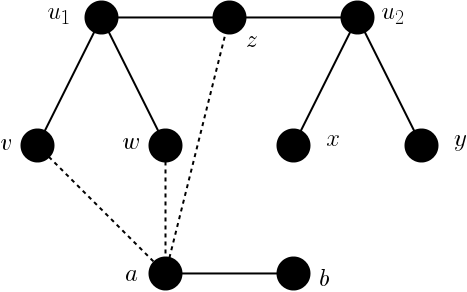}
\label{fig lemma-4-2-b}
\end{figure} 

Now, we look to $b$'s neighbor(s): 
\begin{itemize}
\item If both $bx,by\in E$, then $b$ and $u_1$ satisfy the hypotheses of Lemma \ref{lem case 1} and we're done. 
\item If only one of $\{bx,by\}$ is present, let's say $bx$, then send the cop from $u_2$ to $x$. In response, the robber must be at $a$. Then move the cop from $x$ to $b$ and the robber is trapped. 
\item If neither of $\{bx,by\}$ is present, then the neighbors of $a$ and $b$ (besides $a,b$ themselves) are all adjacent to $u_1$, so we leave the cop at $u_1$ to prevent the robber from leaving the subgraph induced by $a$ and $b$. We then move the cop at $u_2$ along any path from to $a$. The robber will be caught there or at $b$. 
\end{itemize} 

Now, the only way to avoid the above scenario is for $a$ to have exactly one neighbor in $\{v,w\}$ and exactly one neighbor in $\{x,y\}$, and the same for $b$. (If, for instance, $az\in E$, then we cannot avoid the above scenario, regardless of $a$'s other potential neighbor, since $z$ is a common neighbor to both $u_1$ and $u_2$.) Without loss of generality, let's say $av,ax\in E$. (See Figure \ref{fig lemma-4-2-c}. As explained below, $bw,by\in E$ are forced.) 

\begin{figure}[h]
\caption{The remaining scenario when $u_1,u_2$ have exactly one neighbor in common (and not each other) and $ab\in E$.}
\centering 
\includegraphics[width=0.4\textwidth]{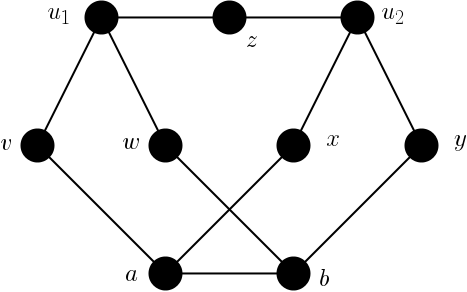}
\label{fig lemma-4-2-c}
\end{figure} 

Consider $b$'s neighbors. Suppose one of them is also a neighbor of $a$; let's say $bv\in E$. Then $v$ and $u_2$ satisfy the hypotheses of Lemma \ref{lem case 1} and we're done. So, in fact, $bw,by\in E$ are required. 

Observe that $vw\in E$ implies that $v$ and $u_2$ satisfy the hypotheses of Lemma \ref{lem case 1} and we're done; likewise for $xy\in E$ and the vertices $y$ and $u_1$. Similarly, observe that $vz\in E$ implies that $b$ and $z$ satisfy the hypotheses of Lemma \ref{lem case 1} and we're done; likewise for $wz\in E$ (vertices $a$ and $z$), $xz\in E$ (vertices $b$ and $z$), and $yz\in E$ (vertices $a$ and $z$). Thus, $z$ has no more neighbors (just $u_1,u_2$) and the only possible edges that could be added to Figure \ref{fig lemma-4-2-c} are $\{vx,vy,wx,wy\}$. 

Suppose $vx,vy\notin E$. Then two lazy cops can win by pushing the robber to $v$ as follows. Start the cops at $u_1,u_2$. The robber can start at either $a$ or $b$. Move the cop from $u_1$ to $w$. In response, the robber must end up at $a$ or $v$. In either case, send the cop from $u_2$ to $x$. In response, the robber must end up at $v$. Send the cop from $x$ to $a$ and the robber is trapped.

This argument also applies when $wx,wy\notin E$: two lazy cops can push the robber to $w$ and win. This means $v$ has another neighbor and so does $w$. Since $\Delta\leq 3$, they cannot have the same neighbor. So, we either have the edges $vx,wy\in E$ present, or else we have the edges $vy,wx\in E$ present.

Suppose $vy,wx\in E$ are the edges present. (See Figure \ref{fig lemma-4-2-d}.) Two lazy cops can win as follows. Start the cops at $u_1$ and $u_2$. 
\begin{itemize}
\item Say the robber starts at $a$. Move the cop from $u_1$ to $v$, forcing the robber to $b$. Move the cop from $u_2$ to $x$, forcing the robber to stay put. Move the cop from $v$ to $y$, trapping the robber. 
\item Say the robber starts at $b$. (This argument is symmetric to the robber starting at $a$.) Move the cop from $u_2$ to $y$, forcing the robber to $a$. Move the cop from $u_1$ to $w$, forcing the robber to stay put. Move the cop from $y$ to $v$, trapping the robber. 
\end{itemize} 

\begin{figure}[h]
\caption{When the edges $vy,wx\in E$ are present.}
\centering 
\includegraphics[width=0.4\textwidth]{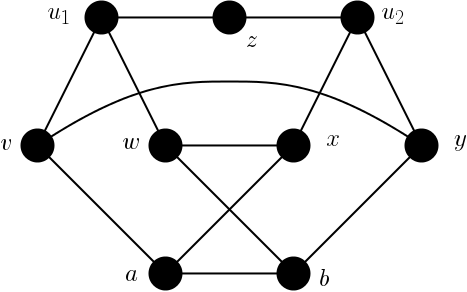}
\label{fig lemma-4-2-d}
\end{figure} 

Instead, suppose $vx,wy\in E$. (See Figure \ref{fig lemma-4-2-e}.) In this case, the vertices $w$ and $x$ satisfy the hypotheses of Lemma \ref{lem case 1} and we're done.

\begin{figure}[h]
\caption{When the edges $vx,wy\in E$ are present.}
\centering 
\includegraphics[width=0.4\textwidth]{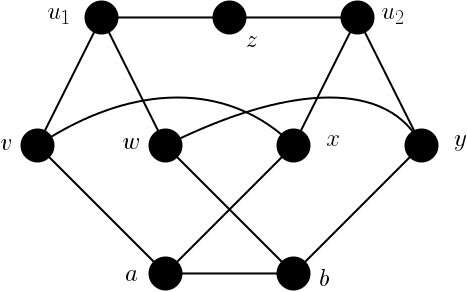}
\label{fig lemma-4-2-e}
\end{figure} 
\end{enumerate}
This completes the proof.
\end{proof}

Admittedly, this proof of Lemma \ref{lem case 2} is the most tedious of this section, but the result will be extremely useful in proving the forthcoming lemmas. Indeed, we attempted to prove some of the lemmas below without proving this result first and found them to be just as tedious as this one was. 

\begin{lem}
\label{lem case 3}
If there exist vertices $u_1,u_2$ both with degree 3 such that $\overline{N}(u_1)\cap \overline{N}(u_2)=\{u_1,u_2\}$, then $c_L(G)\leq 2$. 
\end{lem} 
\begin{proof} 
This scenario implies that $u_1u_2\in E$ but those vertices share no other neighbors. Let's say $N(u_1)=\{u_2,v,w\}$ and $N(u_2)=\{u_1,x,y\}$; there are \textit{exactly three} other vertices, say $a,b,c$.

Observe that none of $\{a,b,c\}$ may have degree 3, for this leads to a scenario compared by the previous Lemmas \ref{lem case 1} and \ref{lem case 2}. If $a$ has degree 3, then it cannot be adjacent to $u_1$ or $u_2$ (which already have full degree 3), so the only way to avoid $u_1$ and $a$ satisfying the hypotheses of Lemma \ref{lem case 1} or \ref{lem case 2} is to have $av,aw\in E$ both present. This means $a$ and $u_2$ have either 0 or 1 neighbors in common, and so they satisfy the hypotheses of either Lemma \ref{lem case 1} or \ref{lem case 2}. 

Thus, $\deg(a)=\deg(b)=\deg(c)=2$. Define $\deg'$ to count the adjacencies amongst only $\{a,b,c\}$. Suppose $\deg'(a)=\deg'(b)=\deg'(c)=0$. Start the cops at $u_1$ and $u_2$, forcing the robber to start (without loss of generality) at $a$. If $N(a)\subseteq N(u_1)$, fix the cop at $u_1$ to prevent the robber from moving, and send the other cop along any path to $a$. Otherwise, $a$ has 1 neighbor in common with $u_1$ and 1 neighbor in common with $u_2$, say $x$. Again, fix the cop at $u_1$ to prevent the robber from moving, and send the other cop along the particular path $u_2-x-a$. 

Consider the induced subgraph $H:=G[\{a,b,c\}]$. From above, we deduce that $H$ has at least one edge. Also, $H$ cannot have three edges: if so, then $a,b,c$ all have full degree 2 with no possibility of edges from $H$ to $G-H$, meaning $G$ is actually disconnected. So, for now, let's say $ab$ is the only edge in $H$. (See Figure \ref{fig lemma-4-3-a}.) Start the cops at $u_1$ and $u_2$; by the argument in the previous paragraph, the robber would certainly not choose to start at $c$, so let's say he starts at $a$. We know $a$ and $b$ each have one neighbor amongst $\{v,w,x,y\}$. If those neighbors are both adjacent to (without loss of generality) $u_1$, then we fix the cop at $u_1$ to prevent the robber from leaving $\{a,b\}$ and send the other cop along any path towards $a$, trapping the robber there or at $b$. If, instead, we have something like $av,bx\in E$, we fix the cop at $u_1$ and send the other cop along the particular path $u_2-x-b-a$, trapping the robber somewhere. 

\begin{figure}[h]
\caption{Cases based on whether there are 1 or 2 edges amongst $\{a,b,c\}$.}
\centering 
\begin{subfigure}[b]{0.4\textwidth}
\includegraphics[width=\textwidth]{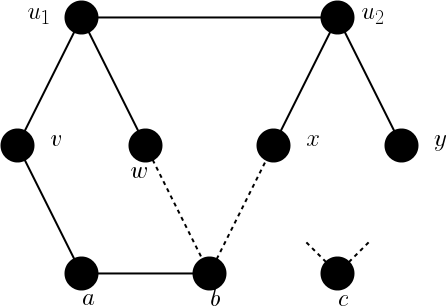}
\caption{One edge amongst $\{a,b,c\}$.}
\label{fig lemma-4-3-a}
\end{subfigure}
\hfill 
\begin{subfigure}[b]{0.4\textwidth}
\includegraphics[width=\textwidth]{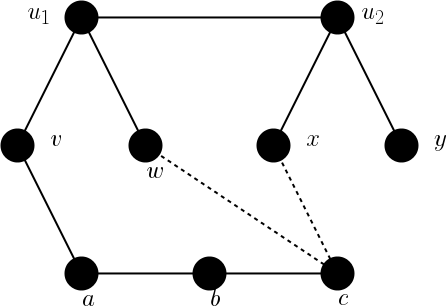}
\caption{Two edges amongst $\{a,b,c\}$.}
\label{fig lemma-4-3-b}
\end{subfigure}
\label{fig lemma-4-3}
\end{figure} 

Now, if there are two edges in $H$, let's say they're $ab,bc$. What we have is the same scenario as described in the previous paragraph, except the path $a-b$ has been subdivided into the path $a-b-c$. (See Figure \ref{fig lemma-4-3-b}.) If the neighbor of $a$ and the neighbor of $c$ are both adjacent to $u_1$, say, then we fix the cop there and send the other cop along any path to the robber. Otherwise, we fix the cop at $u_1$ and send the other cop along the particular path, say $u_2-x-c-b-a$. In any case, the robber is caught by 2 cops. 
\end{proof} 

\begin{lem}
\label{lem case 4}
If there exist vertices $u_1,u_2$ both with degree 3 such that $\overline{N}(u_1)\cap \overline{N}(u_2)=\{y,z\}$ where $y,z\neq u_1,u_2$ (i.e. $u_1$ and $u_2$ have exactly two distinct neighbors in common), then $c_L(G)\leq 2$
\end{lem}
\begin{proof}
This scenario implies that there are \textit{exactly three} vertices, say $a,b,c$, external to $\overline{N}(u_1)\cup\overline{N}(u_2)$. Let's say $N(u_1)=\{x,y,z\}$ and $N(u_2)=\{w,y,z\}$. 

Consider the induced subgraph $H:=G[\{a,b,c\}]$.  Assume that $\deg(a)=3$.  In order for $a$ and $u_1$ to avoid the scenarios covered by Lemmas \ref{lem case 1} and \ref{lem case 2} and \ref{lem case 3}, and since $au_1\notin E$, they must have at least 2 neighbors in common. Thus, either $ay\in E$ or $az\in E$. If $ay\in E$, then $a$ and $y$ satisfy the hypotheses of Lemma \ref{lem case 3}  (and similarly for $az\in E$). Thus, we may assume that $\deg(a)=\deg(b)=\deg(c)=2$. 

There cannot be any edges from $\{a,b,c,x,w\}$ to $\{y,z\}$. Suppose otherwise, say $z$ has a neighbor in $\{a,b,c,x,w\}$. Then $z$ and $u_1$ satisfy the hypotheses of Lemma \ref{lem case 2} and we're done. Now, the edge $yz$ may or may not be present. This will not affect the remainder of the argument. 

Next, consider $x$ and $w$. If $\deg(x)=3$, say, then $x$ and $u_2$ satisfy the hypotheses of Lemma \ref{lem case 1} or \ref{lem case 2}, depending on whether $xw\in E$, and we're done. So, $\deg(x)=\deg(w)=2$. Then, in order for $G$ to be connected (because of $a,b,c$), it must be that $xw\notin E$.

So, $x$ must have a neighbor amongst $\{a,b,c\}$; let's say it's $a$. Then $aw\notin E$, otherwise $G$ would not be connected (because of $b,c$). So, let's say $a$'s other neighbor is $b$. Then, for the same reason $bw\notin E$ (otherwise $c$ is disconnected). This forces the particular graph shown in Figure \ref{fig lemma-4-4} (with $yz$ indicated by a dashed line since it may or may not be present). 

\begin{figure}[h]
\caption{The only graph under consideration (modulo the potential edge $yz$). }
\centering 
\includegraphics[width=0.4\textwidth]{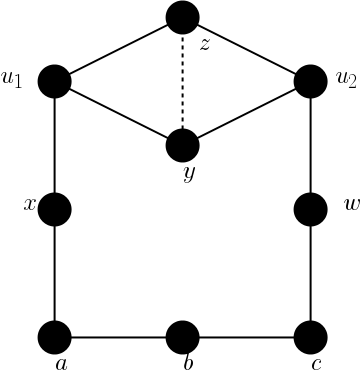}
\label{fig lemma-4-4}
\end{figure} 

Start the cops at $u_1,u_2$, forcing the robber to start amongst $\{a,b,c\}$. Wherever the robber is, fix the cop at $u_1$ and send the other cop along the particular path $u_2-w-c-b-a-x$. He will trap the robber at some point.
\end{proof}

The remaining three cases to be considered are quite short, now that we have completed a majority of the overall cases. 

\begin{lem}
\label{lem case 5}
If there exist vertices $u_1,u_2$ both with degree 3 such that $|N(u_1)\cap N(u_2)|=3$ (i.e. $u_1$ and $u_2$ have all 3 neighbors in common), then $c_L(G)\leq 2$.
\end{lem}
\begin{proof}
Consider the mutual neighborhood of $u_1$ and $u_2$, say $x,y,z$. There are 4 other vertices in the graph, say $a,b,c,d$. Since $u_1$ and $u_2$ cannot have any more neighbors, at least one of $\{x,y,z\}$ is adjacent to $\{a,b,c,d\}$. Let's say $ax\in E$. Then, $x$ and $u_1$ satisfy the hypotheses of Lemma \ref{lem case 3} and we're done. Otherwise, we are left with a disconnected graph.
\end{proof}

\begin{lem}
\label{lem case 6}
If there exist vertices $u_1,u_2$ both with degree 3 such that $\overline{N}(u_1)\cap \overline{N}(u_2)=\{u_1,u_2,y,z\}$ (i.e. $u_1$ and $u_2$ have two neighbors in common in addition to being adjacent), then $c_L(G)\leq 2$.
\end{lem}
\begin{proof}
Consider the vertices not adjacent to either $u_1$ or $u_2$ and call them $a,b,c,d$. Now, if $\deg(a)=3$, then we cannot avoid $a$ and $u_1$ satisfying the hypotheses of Lemma \ref{lem case 1} or \ref{lem case 2} or \ref{lem case 4} since they have at most two neighbors in common. Thus, $\deg(a)=\deg(b)=\deg(c)=\deg(d)=2$. 

Next, we must have either $\deg(x)=3$ or $\deg(y)=3$ (or possibly both). This is because we currently have only two vertices of degree 3; if this remains, then $G$ has only 10 edges, in which case the argument in Section \ref{sec 10 edges} applies. So, there must be two vertices of degree 3 amongst the candidates $\{x,y,z\}$ (since, overall in $G$, the number of vertices of degree 3 is even). Let's say $\deg(x)=3$. To avoid having $x$ and $u_1$ satisfying the hypotheses of any of the previous Lemmas in this section, we must have $xz\in E$. (Indeed, this makes $x$ and $u_1$ satisfy the hypotheses of this very lemma we're proving.) 

\begin{figure}[h]
\caption{$x$ and $u_2$ must satisfy a previous lemma}
\centering 
\includegraphics[width=0.4\textwidth]{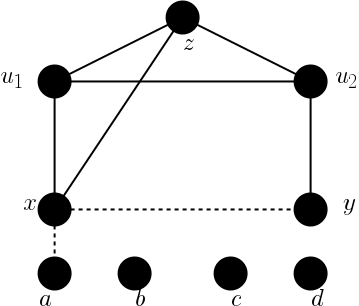}
\label{fig lemma-4-6}
\end{figure} 

Now, $x$ has one other neighbor. If $xy\in E$ then $x$ and $u_2$ satisfy the hypotheses of Lemma \ref{lem case 4}. Otherwise, $x$ and $u_2$ satisfy the hypotheses of Lemma \ref{lem case 2}. See Figure \ref{fig lemma-4-6}
\end{proof}

\begin{lem}
\label{lem case 7}
If there exist vertices $u_1,u_2$ both with degree 3 such that $\overline{N}(u_1)\cap \overline{N}(u_2)=\{u_1,u_2,y,z\}$ (i.e. $u_1$ and $u_2$ have two neighbors in common in addition to being adjacent), then $c_L(G)\leq 2$.
\end{lem}
\begin{proof}
This scenario implies that there are five vertices outside the collective neighborhood $\{u_1,u_2,y,z\}$; call these $\{a,b,c,d,e\}$. To ensure $G$ is connected, there must be some edge between those two sets. Since $u_1,u_2$ cannot have more neighbors, let's say $az\in E$. Now, $z$ and $u_1$ have each other and exactly one more neighbor in common, so they satisfy the hypotheses of Lemma \ref{lem case 6} and we're done. 
\end{proof}

This completes the case-by-case analysis for graphs on 9 vertices with $11\leq |E|\leq 13$. Overall, we have actually proven the main result, as explained in the next section. 

\section{Summary and future work}
\label{sec summary}
\subsection{Main result}
Here, we restate \textbf{Theorem \ref{thm R3}} and explain how it has been proven: 
\begin{thm*}[\ref{thm R3}]
The $3\times 3$ Rooks graph $G=K_3\square K_3$ is the unique graph on 9 vertices with $c_L(G)=3$. All other graphs $H$ on 9 vertices have $c_L(H)\leq 2$.
\end{thm*}
In Section \ref{sec large degree}, we showed that any graph on $|V|\leq 8$ vertices has $c_L\leq 2$. Then, we considered graphs on 9 vertices with maximum degree $\Delta\geq 4$. We found that the only way to avoid creating a graph on which two lazy cops can win was to construct the particular graph $R_3$. 

It remained to show that all graphs on 9 vertices with $\Delta\leq 3$ also have $c_L\leq 2$. By Lemma \ref{lem delta 2}, we needed only to consider graphs that also have $\delta\geq 2$, which narrowed our search to $8\leq |E|\leq 13$. For $|E|=8,9,10$, we could analyze these graphs easily. For $11\leq |E|\leq 13$, we looked to vertices of degree three (of which there must be at least four, in fact) and considered case-by-case scenarios based on their common neighbors. In every case, we could either win with two lazy cops or reduce the scenario to a previous one. 

By virtue of Theorem \ref{thm Rn}, we also know $c_L(R_3)=3$. Overall, this has shown the main result. The only graph on 9 vertices satisfying $c_L=3$ is that particular graph, and any smaller graph has $c_L\leq 2$. 

\subsection{Conjectures and partial progress}
\begin{conj}
\label{conj Rn}
The unique smallest graph for which $c_L=n$ is $R_n=K_n\square K_n$.
\end{conj}
This article has proven the $n=3$ case, and the $n=1,2$ cases are obvious. We also know, from Theorem \ref{thm Rn}, that $c_L(R_n)=n$. So, this conjecture seems reasonable.

Before concluding, we share some ideas that may make some progress towards proving this conjecture for the $n=4$ case, at least.  We start with a corollary to Theorem \ref{thm 8 vertices} that generalizes Lemma \ref{lem Delta n-4}: 
\begin{cor}
\label{cor delta n-9}
For a graph $G$ on $n$ vertices with $\Delta\geq n-9$, we have $c_L(G)\leq 3$.
\end{cor}
\begin{proof}
Consider a vertex $u$ with $\deg(u)\geq n-9$ and all of its neighbors, and consider removing these from the graph. Let $H$ be the induced subgraph on what remains. This graph $H$ has at most 8 vertices, so Theorem \ref{thm 8 vertices} guarantees that 2 lazy cops can win on that subgraph. In the overall graph $G$, start a cop at $u$ and keep him there throughout the game. Play the winning strategy on $H$ with the other two cops. 
\end{proof}
The following conjecture naturally generalizes these ideas. Its veracity would follow inductively, were Conjecture \ref{conj Rn} proven true, using an argument like the one in the previous proof. Conversely, we wonder whether Conjecture \ref{conj Rn} is logically equivalent to this one.
\begin{conj}
For a graph $G$ on $n$ vertices with $\Delta\geq n-k^2$, we have $c_L(G)\leq k$.
\end{conj}

We can, at the very least, use Corollary \ref{cor delta n-9} to narrow our search for graphs with $c_L\geq 4$, in much the same way that we used Lemma \ref{lem Delta n-4} to narrow our search for graphs with $c_L\geq 3$. Suppose we have a graph with $|V|\leq 15$: 
\begin{itemize}
\item If $\Delta\geq 6$, then Corollary \ref{cor delta n-9} applies.
\item If $\Delta=5$, then take a vertex $u$ with $\deg(u)=5$ and look at the subgraph $H$ induced by the vertices external to $\overline{N}(u)$; $H$ has at most nine vertices, so either it has $c_L\leq 2$ or else it is precisely the Rooks graph $R_3$. If $c_L(H)\leq 2$, just station a cop at $u$ and let two others play on $H$ to win. If $H$ is precisely $R_3$, then to ensure $G$ is connected, there must be an edge from some vertex in $H$, say $a$, to one of $u$'s neighbors, say, $b$. But, to then ensure $\Delta=5$, and knowing $R_3$ is 4-regular, that vertex in $R_3$ \textit{cannot have any more neighbors}. So, start a cop at $u$ and start two cops on the copy of $R_3$ to force the robber to start at $a$. Send the cop from $u$ to $b$ and the robber is both threatened and trapped. 
\item If $\Delta=4$, then take a vertex with $\deg(u)=4$ and look at the subgraph $H$ induced by the vertices external to $\overline{N}(u)$; $H$ has at most ten vertices. If $c_L(H)\leq 2$, then we're good. If $c_L(H)=3$, though \ldots this is where things get interesting. At this point, we hope for a characterization of those graphs with 10 vertices that require 3 lazy cops. Based on Beveridge et al. and the result contained in the present article, we believe that the only such graphs are the Petersen Graph and a copy of $R_3$ with an additional vertex (subject to some conditions presently unknown). We postpone any further progress on this case and, in the meantime, are seeking to prove this conjecture, which would be quite helpful in this step: 
\begin{conj}
The only graphs $G$ on 10 vertices with  $c_L(G)=3$ are 
\begin{enumerate}
\item the Petersen Graph, or 
\item a copy of $R_3$ with an additional vertex $v$ whose degree is at most 5, and whose adjacencies are chosen in a particular way so as not to inadvertently reduce the lazy cop number to two. 
\end{enumerate}
\end{conj} 
\item If $\Delta=3$, then we feel strangely stuck \ldots  We believe this shouldn't be so challenging for the cops because of how few edges there could be. However, this seems to make it inordinately difficult to find concise, general arguments. A better understanding of graphs with $\Delta=3$, in general, would be extremely helpful. 
\end{itemize} 
We welcome suggestions and results relevant to any of these conjectures.

\bibliography{mybibfile-R3} 

\begin{thebibliography}{10}
\expandafter\ifx\csname url\endcsname\relax
  \def\url#1{\texttt{#1}}\fi
\expandafter\ifx\csname urlprefix\endcsname\relax\def\urlprefix{URL }\fi
\expandafter\ifx\csname href\endcsname\relax
  \def\href#1#2{#2} \def\path#1{#1}\fi

\bibitem{nowakowski1983vertex}
R.~Nowakowski, P.~Winkler, Vertex-to-vertex pursuit in a graph, Discrete
  Mathematics 43~(2-3) (1983) 235--239.

\bibitem{quilliot1978jeux}
A.~Quilliot, Jeux et pointes fixes sur les graphes, Ph.D. thesis, Ph. D.
  Dissertation, Universit{\'e} de Paris VI (1978).

\bibitem{bonato2011game}
A.~Bonato, R.~J. Nowakowski, The game of cops and robbers on graphs, Vol.~61,
  American Mathematical Society Providence, 2011.

\bibitem{aigner1984game}
M.~Aigner, M.~Fromme, A game of cops and robbers, Discrete Applied Mathematics
  8~(1) (1984) 1--12.

\bibitem{baird2013minimum}
W.~Baird, A.~Beveridge, A.~Bonato, P.~Codenotti, A.~Maurer, J.~McCauley,
  S.~Valeva, On the minimum order of k-cop-win graphs, arXiv preprint
  arXiv:1308.2841.

\bibitem{beveridge2011petersen}
A.~Beveridge, P.~Codenotti, A.~Maurer, J.~McCauley, S.~Valeva, The petersen
  graph is the smallest 3-cop-win graph, arXiv preprint arXiv:1110.0768.

\bibitem{baird2013meyniel}
W.~Baird, A.~Bonato, Meyniel's conjecture on the cop number: a survey, arXiv
  preprint arXiv:1308.3385.

\bibitem{offner2014variations}
D.~Offner, K.~Okajian, Variations of cops and robber on the hypercube,
  Australasian Journal of Combinatorics 59~(2) (2014) 229--250.

\bibitem{bal2015lazy}
D.~Bal, A.~Bonato, W.~B. Kinnersley, P.~PRA{\L}AT, Lazy cops and robbers on
  hypercubes, Combinatorics, Probability and Computing 24~(06) (2015) 829--837.

\bibitem{bal2015lazy2}
D.~Bal, A.~Bonato, W.~B. Kinnersley, Lazy cops and robbers played on random
  graphs and graphs on surfaces, Preprint.

\bibitem{kinnersley2015cops}
W.~B. Kinnersley, Cops and robbers is exptime-complete, Journal of
  Combinatorial Theory, Series B 111 (2015) 201--220.

\bibitem{imrich2000product}
W.~Imrich, S.~Klavzar, Product graphs, Wiley, 2000.

\bibitem{hammack2011handbook}
R.~Hammack, W.~Imrich, S.~Klav{\v{z}}ar, Handbook of product graphs, CRC press,
  2011.

\bibitem{yaglom1987challenging}
A.~Yaglom, I.~Yaglom, Challenging Mathematical Problems with Elementary
  Solutions, Dover, 1987.

\bibitem{berarducci1993cop}
A.~Berarducci, B.~Intrigila, On the cop number of a graph, Advances in Applied
  Mathematics 14~(4) (1993) 389--403.

\bibitem{lakatos2015proofs}
I.~Lakatos, Proofs and refutations: The logic of mathematical discovery,
  Cambridge university press, 2015.

\end{thebibliography}
\end{document}